\documentclass[12pt,english,reqno]{amsart} 
\usepackage[T1]{fontenc}
\usepackage[latin9]{inputenc} 
\usepackage{amsthm} 
\usepackage{amstext}
\usepackage{graphicx} 
\usepackage{amssymb} 
\usepackage[all]{xy}

\makeatletter
\numberwithin{equation}{section}
\numberwithin{figure}{section}
\theoremstyle{plain}
\newtheorem{thm}{Theorem}
  \theoremstyle{remark}
  \newtheorem{rem}[thm]{Remark}
  \theoremstyle{definition}
  \newtheorem{defn}[thm]{Definition}
  \theoremstyle{plain}
  \newtheorem{cor}[thm]{Corollary}
 \theoremstyle{definition}
  \newtheorem{example}[thm]{Example}

\usepackage{color}
\usepackage{amsthm}
\usepackage{latexsym}
\usepackage{amscd}
\usepackage{array}

\newcommand{\N}{{\mathbb N}}
\newcommand{\Q}{{\mathbb Q}}
\newcommand{\R}{{\mathbb R}}
\newcommand{\LL}{{}^*\R_\text{f}}

\newcommand{\st}{{\rm st}}

\newcommand{\ER}{{^\bullet\R}}

\newcommand{\xyR}[1]{ \makeatletter
\xydef@\xymatrixrowsep@{#1} \makeatother}\makeatother
\newcommand{\ext}[1]{{}^*{#1}_{\text{\rm f}}}
\newcommand{\extF}[1]{{}^\bullet #1}
\newcommand{\stF}[1]{{^\circ #1}}
\newcommand{\diff}[1]{\,{\rm d}#1}
\makeatother

\makeatother

\begin{document}

\thispagestyle{empty}

\title[Does Cantor factor through infinitesimals?]{Two ways of
obtaining infinitesimals by refining Cantor's completion of the reals}

\author{Paolo Giordano}

\thanks{P. Giordano was supported by a L. Meitner FWF (Austria; grant
M1247-N13).}

\address{Department of Mathematics, University of Vienna, Austria.}

\email{paolo.giordano@univie.ac.at}

\author{Mikhail G. Katz}

\thanks{M.G. Katz was supported by the Israel Science Foundation
(grant no.~1294/06).}

\address{Department of Mathematics, Bar Ilan University, Ramat Gan,
52900 Israel.}

\email{katzmik@macs.biu.ac.il}
\begin{abstract}
Cantor's famous construction of the real continuum in terms of Cauchy
sequences of rationals proceeds by imposing a suitable equivalence
relation. More generally, the completion of a metric space starts from
an analogous equivalence relation among sequences of points of the
space. Can Cantor's relation among \emph{Cauchy} sequences of reals be
refined so as to produce a Cauchy complete and infinitesimal-enriched
continuum? We present two possibilities: one leads to invertible
infinitesimals and the hyperreals; the other to nilpotent
infinitesimals (e.g.  $h\ne0$ infinitesimal such that $h^{2}=0$) and
Fermat reals.  One of our themes is the trade-off between formal power
and intuition.
\end{abstract}


\maketitle
\tableofcontents{}

\section{Introduction}

\label{one}

In a recent issue of \emph{The American Mathematical Monthly\/},
K.~Hrbacek \emph{et al} have argued that analysis needs better
axiomatics than Zermelo-Fraenkel set theory with the axiom of choice
(ZFC), see \cite{HLO}.  They propose a new axiomatic framework that
naturally includes numbers they call \emph{ultrasmall} (i.e.,
infinitesimal). They mention \cite[p.~803]{HLO} that these ideas are
not entirely new, and provide a list of references the earliest of
which is E. Nelson's 1977 text \cite{Ne}, where the author outlined
his enrichment of ZFC known as Internal Set Theory (IST).

While the axiomatic approach has much to recommend itself, we feel
that a prerequisite for new axiomatics is a good understanding of the
mathematical structure that stands to be axiomatized, and not vice
versa. To make a convincing case in favor of a new axiom system, one
first needs to explain the basics. In the case of
infinitesimal-enriched continua, the basics amount to understanding
the ultrapower construction.  One of our goals in this text is to give
an accessible explanation of the latter in the context of Cauchy
sequences, as well as providing possible alternatives.

Cantor's completion of the rationals resulting in the field of real
numbers proceeds by quotienting the space
$\mathcal{C}_\Q\subset\Q^{\N}$ of all Cauchy sequences of rational
numbers by Cauchy's equivalence relation. Similarly, the collection
$\mathcal{C}\subset\R^{\N}$ of all Cauchy sequences of real numbers
projects to the Archimedean continuum $\R$: \begin{equation}
\mathcal{C}\xrightarrow{\lim}\R.\label{C1}\end{equation} The
corresponding equivalence relation $\sim_{\mathcal C}$ defined by
\[
u\sim_{\mathcal{C}}v\text{\;\; \;if and only
if\;\;}\lim_{n\to\infty}\left|u_{n}-v_{n}\right|=0,
\] 
``collapses'' all null sequences to a single point $0\in\R$.  Is there
another way to define an equivalence relation $\sim$ on the
space~$\mathcal{C}$ that would allow some null sequences to retain
their distinct identity? In other words, can one {\em refine\/}
Cantor's equivalence relation among Cauchy sequences, in such a way as
not to {}``collapse'' all null sequences to zero? The idea would be
that, relative to a new equivalence relation $\sim$, a null sequence
of reals would become an actual infinitesimal. In other words, we are
searching for a new notion of {}``completion'', with respect to which
the real field~$\R$ can be completed by the addition of
infinitesimals. What one seeks is an intermediate stage,
$\LL:=\mathcal{C}/\sim$, in the projection~\eqref{C1}.  The subscript
{}``f'' in the symbol $\LL$ stands for {}``finite'' (i.e., there are
no infinite numbers).  Such an intermediate stage $\LL$ would
represent an infinitesimal-enriched continuum as in
Figure~\ref{helpful}.

\begin{figure}
\[
\xyR{1.2cm}\xymatrix{ & &
^{*}\R_{\text{f}}\ar@<0.5ex>[d]^{\text{st}}\\
\mathcal{C}\;\ar[rr]_{\lim}\ar[urr]^{p} & & \R\ar@{^{(}->}@<0.5ex>[u]}
\]
\caption{\textsf{Factoring Cantor's map $\mathcal{C}\to\R$} through
an intermediate stage $^{*}\R_{\text{f}}$.}
\label{helpful} 
\end{figure}

Here, if $[u]_{\sim}$ is the new equivalence class of a sequence $u$,
then the function \[ \text{st}:\LL\to\R,\] defined by \[
\text{st}([u]_{\sim}):=\lim_{n\to+\infty}u_{n} \; \in\R\] is the usual
limit of a Cauchy sequence
$u=\left(u_{n}\right)_{n\in\N}\in\mathcal{C}$.  This function
represents the \emph{standard part} of $[u]_{\sim}\in\LL$, that is a
standard real number infinitely close to the new number
$[u]_{\sim}\in\LL$. The most natural way to obtain a ring structure on
$\LL$ is to define the equivalence relation $\sim$ so that it
preserves pointwise sums and products. Therefore, we expect $\LL$ to
be a ring rather than a field, because it cannot contain the pointwise
inverse $\left(\frac{1}{u_{n}}\right)_{n\in\N}$ of an infinitesimal
$[u_{n}]$, since the inverse is not a Cauchy sequence. 

In this text, we will explore
two possible implementations of these ideas.

\section{A possible approach with invertible infinitesimals}

To implement the ideas outlined in Section~\ref{one}, a possible
approach is to declare two Cauchy sequences $u$, $v\in\mathcal{C}$ to
be equivalent if they coincide on a {}``dominant'' set of indices in
$\N$: \begin{equation} u\sim v\quad\iff\quad\left\{ n\in\N\,|\,
u_{n}=v_{n}\right\} \text{ is
dominant.}\label{eq:idea_dominant}\end{equation} For simplicity, we
will use the symbol $[u]$ for the equivalence class $[u]_{\sim}$
generated by $u\in\mathcal{C}$.

What is {}``dominant''? A finite set in $\N$ is never dominant; every
cofinite set (i.e., set with finite complement) is necessarily
dominant, and we also expect the property that the superset of a
dominant set is dominant, as well. Moreover, we expect the relation
\eqref{eq:idea_dominant} to yield an equivalence relation.  In
particular, the validity of the transitive property for generic Cauchy
sequences implies that the intersection of two dominant sets is
dominant. In fact, let us assume that 
\begin{equation} 
\forall u,v,w\in\mathcal{C}: \quad u\sim v\ \wedge\ v\sim
w\;\Rightarrow\; u\sim w.\label{eq:transitiveByHyp}
\end{equation}
Then, if sets $A$ and $B$ of indices are dominant, it suffices to
take%
\footnote{Let us note that in $\N$ we have $0\in\N$.}
\[
u_{n}:=\begin{cases} 1 & \text{if }n\in A\\ 1-\frac{1}{n+1} & \text{if
}n\in\N\setminus A\end{cases}\qquad w_{n}:=\begin{cases} 1 & \text{if
}n\in B\\ 1+\frac{1}{n+1} & \text{if }n\in\N\setminus B\end{cases}\]
to obtain $u\sim1$ and $1\sim w$, so that $u\sim w$
from~\eqref{eq:transitiveByHyp}.  It follows that the set $\left\{
n\in\N\,|\, u_{n}=w_{n}\right\} =A\cap B$ is dominant.  Conversely, if
our family of dominant sets is closed with respect to finite
intersections, then the relation $\sim$ is an equivalence
relation. For example, the family of all cofinite sets
\[ 
\mathcal{F}:=\left\{ S\subseteq\N \;|\; \N\setminus S\text{ is
finite}\right\} ,\] the so-called Fr\'echet filter, satisfies all the
conditions we have imposed so far on dominant sets.  These conditions
define the notion of a \emph{filter} on the set $\N$ (extending the
Fr\'echet filter).

It is easy to prove that the equivalence relation $\sim$ preserves
pointwise operations 
\begin{equation}
[u]+[v]:= \left[\left(u_{n}+v_{n}\right)_{n\in\N}\right]\quad
\text{and} \quad [u]\cdot[v]:=\left[\left(u_{n}\cdot
v_{n}\right)_{n\in\N}\right]
\label{eq:pointwiseRingOpearations}
\end{equation}
so that the quotient 
\begin{equation}
\label{24}
\LL:=\mathcal{C}/\!\!\!\sim
\end{equation}
becomes a ring.  Moreover, the real numbers $\R$ are embedded in $\LL$
as constant sequences.  Whether or not $\LL$ is an integral domain
depends on the choice of the filter of dominant sets. Thus, the
product of sequences $u$ and~$w$ given by \[ u_{n}:=\begin{cases} 0 &
\text{if }n\text{ is even}\\ \frac{1}{n+1} & \text{if }n\text{ is
odd}\end{cases}\qquad w_{n}:=\begin{cases} \frac{1}{n+1} & \text{if
}n\text{ is even}\\ 0 & \text{if }n\text{ is odd}\end{cases}\] is
zero, but whether or not $[u]$ is zero depends on whether the set of
even numbers is considered to be dominant or not. We will solve this
problem later.

To show that the relation as in \eqref{eq:idea_dominant} is a
refinement of the usual Cauchy relation~$\sim_{\mathcal{C}}$, assume
that $u$, $v\in\mathcal{C}$ coincide on a dominant set~$A$.  Then we
have~$u_{\sigma_{n}}-v_{\sigma_{n}}=0$ for some subsequence
$\sigma:\N\rightarrow\N$ (enumerating the members of the set $A$). It
follows that $u\sim_{\mathcal{C}}v$ since $u$ and $v$ converge. Of
course, the relation $\sim$ is a strict refinement because if we take
$u_{n}=\frac{1}{n}$ and $v_{n}=0$, then $u\sim_{\mathcal{C}}v$ but
$\left\{ n\in\N\,|\, u_{n}=v_{n}\right\} =\emptyset$ is the empty set,
which is never dominant.

Whether the idea expressed by the notion of a dominant set as
in~\eqref{eq:idea_dominant} can be considered {}``natural'' or not is
a matter of opinion. An alternative approach would be to define a new
equivalence relation in terms of the rate of convergence of the
difference $u-v$. A thread going in this direction will be presented
in Section~\ref{sub:nilpInf}, but here we will continue with the
approach based on \eqref{eq:idea_dominant}.  If one accepts this idea,
then it is also natural to define an \emph{order}, by setting
\begin{equation} [u]\ge[v]\quad\iff\quad\left\{ n\in\N\,|\, u_{n}\ge
v_{n}\right\} \text{ is dominant}.\label{eq:order}\end{equation} This
yields an ordered ring, as one can easily check.

Is this order total? The assumption that it is total,
i.e. 
\begin{equation} \forall u\in\mathcal{C}:\
[u]\ge0\quad\text{or}\quad[u]\le0,\label{eq:totalOrder}
\end{equation}
yields a further condition on dominant sets. In fact, if $A$ is
dominant, then defining
\[ 
u_{n}:=\begin{cases} \frac{1}{n+1} & \text{if }n\in A\\ -\frac{1}{n+1}
& \text{if }n\in\N\setminus A\end{cases}
\] 
we have that $A$ is dominant if the first alternative of
\eqref{eq:totalOrder} holds; otherwise $\N\setminus A$ is dominant. A
filter satisfying this additional condition is called a \emph{free
ultrafilter}.%
\footnote{For the sake of completeness, we say that
$\mathcal{U}\subseteq\mathcal{P}(I)$ is an ultrafilter on the set $I$
if $\emptyset\notin\mathcal{U}$; $\mathcal{U}$ is closed with respect
to finite intersections and supersets; and $X\in\mathcal{U}$ or
$I\setminus X\in\mathcal{U}$ for every subset $X$ of $I$.}
Using this additional condition, we are now also able to prove that
$\LL$ is an integral domain.
\begin{thm}
$\LL$ is an integral domain. \end{thm}
\begin{proof}
Given nonzero classes $[u]\not=0$ and $[v]\not=0$, both of the sets
$\{n\in\N\,|\, u_{n}\not=0\}$ and $\{n\in\N\,|\, v_{n}\not=0\}$
are dominant. Therefore so is their intersection. 
\end{proof}
Given an integral domain, we can consider the corresponding field of
fractions ${}^{*}\R_{\text{frac}}$. Since $\LL$ is also an ordered
ring, the order structure extends to the quotient field of fractions
in the usual way.
\begin{rem}
In a classical approach to nonstandard analysis, the equality on a
dominant set (Formula \eqref{eq:idea_dominant}) is applied to
\emph{arbitrary} sequences, rather than merely Cauchy sequences.
Nonetheless, our field of fractions ${}^{*}\R_{\text{frac}}$ is
isomorphic to the full hyperreal field%
\footnote{Namely, the field obtained as the quotient of $\R^\N$ using
the same ultrafilter; in general, the result will depend on the
ultrafilter.}
${}^{*}\R$ of nonstandard analysis through
\[
\frac{[u]}{[v]}\in{}^{*}\R_{\text{frac}}\mapsto\left[\left(\frac{u_{n}}{v_{n}}\right)_{n\in\N}\right]_{\mathcal{U}}\in{}^{*}\R,
\]
where $[(q_{n})_{n}]_{\mathcal{U}}$ is the equivalence class modulo
the ultrafilter $\mathcal{U}$. To prove this, note that every sequence
$q\in\R^{\N}$ can be written as $q=\frac{u}{v}$ for two null sequences
$u$, $v$.  Thus, we can
set~$u_{n}:=q_{n}\cdot\frac{1}{e^{|q_{n}|}\cdot(n+1)}$ and
$v_{n}:=\frac{1}{e^{|q_{n}|}\cdot(n+1)}$.
\end{rem}

\section{A free ultrafilter, anyone?}

Our intuition yearns for meaningful examples of free ultrafilters.
Such can be obtained by using Zorn's lemma. It is possible to prove
that some weaker form of the axiom of choice is necessary to prove the
existence of a free ultrafilter.  The use of this axiom in modern
mathematics is routine.  Thus, one of its standard consequences is the
Hahn-Banach theorem, of fundamental importance in functional
analysis.%
\footnote{Luxemburg \cite{Lu69} explored possibilities of constructing extensions
that rely on Hahn-Banach only (rather than full choice).%
} Yet, one consequence of exploiting this axiom is that we don't possess
detailed information about how free ultrafilters are made. Moreover,
this also implies that it is not so easy to prove the existence of
a free ultrafilter satisfying some given and potentially useful conditions. 
\begin{thm}
The Fr\'echet filter can be extended to a free ultrafilter. \end{thm}
\begin{proof}
See Tarski \cite{Tar} (1930). 
\end{proof}
We have to admit that it is not so easy to judge the idea represented
by formula~\eqref{eq:idea_dominant}.  Indeed, starting from our
definition of pointwise operations and order, one can easily guess
that this idea is formally very powerful. For example, it is almost
trivial to extend to $\LL$ the validity of general laws about real
numbers, such as the following law: 
\begin{equation}
\forall x,y\in\R:
\sin(x+y)=\sin(x)\cos(y)+\cos(x)\sin(y).\label{eq:sineSum}
\end{equation}
In fact, we can extend trigonometric functions pointwise. Namely, we
extend $\sin:\R\rightarrow\R$ to ${}^{*}\!\sin:\LL\rightarrow\LL$ by
setting \[
{}^{*}\!\sin\left(\left[u\right]\right):=\left[\left(\sin(u_{n})\right)_{n\in\N}\right].\]
Finally, the law \eqref{eq:sineSum} extends to ${}^{*}\!\sin$ and
${}^{*}\!\cos$ because the set of indices $n\in\N$ where it is true is
all of $\N$. We will deal with this extension of laws from $\R$ to
$\LL$ in more general terms in Section~\ref{sec:Leibniz's-law}.

On the other hand, whatever will be the example of ultrafilter we will
be able to present, it doesn't seems sufficiently meaningful why the
infinitesimal
$\left[\left(\frac{(-1)^{n}}{n+1}\right)_{n\in\N}\right]$ should be
considered positive rather than negative, or vice versa.%
\footnote{Moreover, examining the conditions defining the notion of
ultrafilter, one can guess that the notion of a dominant set is not
intuitively so clear. In point of fact, the technically desirable
conditions about the closure with respect to intersection and
complement can lead to counter intuitive consequences. We would have
that even numbers $P_{2}$ or odd numbers will be dominant (but not
both). Let us suppose, e.g., the first case and continue: even numbers
in $P_{2}$, i.e. the set $P_{4}$ of multiples of 4, or its complement
$\N\setminus P_{4}$ will be dominant. In the latter case, also
$P_{2}\cap(\N\setminus P_{4})$, i.e. numbers of the form $2(2n+1)$,
will be dominant. In any case we would be able to find always a
dominant set which has {}``1/2 of the elements of the previous
dominant set''. Continuing in this way, we can obtain a dominant set,
which is intuitively very {}``thin'' with respect to its
complement. To understand this idea a little better, let us consider
that everything we said up to now can be generalized if instead of
sequences $u:\N\rightarrow\R$ we take functions
$u:[0,1]\rightarrow\R$.  In other words, instead of taking our indices
as integer numbers, we take real numbers in $[0,1].$ Then, we can
repeat the previous reasoning considering, at each step $k$,
subintervals of length $2^{-k}$.  Therefore, for every
$\varepsilon>0$, we are always able to find in an ultrafilter on
$[0,1]$ a dominant set $A$ whose uniform probability $P(A)<\epsilon$,
whereas $P([0,1]\setminus A)>1-\epsilon$, even though this complement
is not dominant. See \cite{Gio09} for a formalization of this idea
using the notion of density of subsets of $\N$.}

To summarize, the idea of requiring sequences to coincide on dominant
sets, even if it may seem initially forbidding from an intuitive point
of view, appears to be formally extremely powerful.  As an
alternative, in Section~\ref{sec:nilpotentInf} we will present another
idea, which is intuitively clear but which doesn't seem equally
powerful.  Which thread one wishes to follow would depend on
applications envisioned.

\section{Actual infinitesimals, null sequences, and standard part}

What are, formally, the infinitesimals in the ring $\LL$
of~\eqref{24}, and how are they related to null sequences? An
infinitesimal is a number in $\LL$ which belongs to every interval of
the form $\left[\frac{-1}{n},\frac{1}{n}\right]$:
\begin{defn}
We say that $x\in\LL$ is \emph{infinitesimal} if and only if \[
\forall n\in\N_{\ne 0}:\ -\frac{1}{n}<x<\frac{1}{n},\] and we will
write $x\approx0$. Similarly, we write $y\approx z$ if
$y-z\approx0$. Clearly, such an $x\in\LL$ will be infinitesimal if and
only if in the field of fractions ${}^{*}\R_{\text{frac}}$, the
element~$x^{-1}$ is infinite,%
\footnote{Note that without additional assumptions,
${}^{*}\!\R_{\text{frac}}$ may contain additional more infinitesimals
not found in~$\LL$; see Remark~\ref{rem:P-poinInNSA} below.}
i.e. it doesn't satisfy a bound of the form $|x|<n$ for some $n\in\N$.
\end{defn}
Do infinitesimals in $\LL$ correspond to ordinary null sequences? 
\begin{thm}
\label{thm:infinitesimalsAndNullSequences}Let $[u]\in\LL$, then
we have that\[
[u]\text{ is infinitesimal}\]
 if and only if\[
\lim_{n}u_{n}=0\]
 \end{thm}
\begin{proof}
\noindent Let us assume that $[u]$ is infinitesimal, then for each
$n\in\N_{\ne 0}$, the set \[ A_{n}:=\left\{
k\in\N\,:\,-\frac{1}{n}<u_{k}<\frac{1}{n}\right\} \] is
dominant. Therefore, it is infinite and we can always find an
increasing sequence $k:\N\rightarrow\N$ such that $k_{n}\in A_{n}$ and
$k_{n+1}>k_{n}$.  For such a sequence we have\[ \forall
n\in\N_{\ne0}:\ -\frac{1}{n}<u_{k_{n}}<\frac{1}{n}.\] Since
$u\in\mathcal{C}$ is a Cauchy sequence, we obtain
\[
\lim_{n\to+\infty}u_{n}=\lim_{n\to+\infty}u_{k_{n}}=0.
\] 
To prove the converse implication, we can consider that\[ \forall
n\in\N_{\ne0}\,\exists N:\ \forall k\in\N_{\ge
N}:\-\frac{1}{n}<u_{k}<\frac{1}{n}.\] Since every cofinite set
$\N_{\ge N}$ is dominant, this proves that $[u]$ is infinitesimal.
\end{proof}
\noindent 
As a corollary, we have that $[u]\approx[v]$ if and only if
$\lim_{n}u_{n}=\lim_{n}v_{n}$. This allows us to define the standard
part mentioned above.
\begin{defn}
Let $[u]\in\LL$, then the real number \begin{equation}
\text{st}([u]):=
\lim_{n\to+\infty}u_{n}\in\R\label{eq:DefStandardPart}\end{equation}
is called the \emph{standard part} of $[u]$.
\end{defn}
\noindent
Note that we have~$x\approx\text{st}(x)$ for every $x\in\LL$.

Is our extension $\LL$ of $\R$ still Cauchy complete with respect to
some kind of metric extending the usual Euclidean metric on $\R$?  It
is not hard to prove that \[
d\left(x,y\right):=\left|\text{st}(x)-\text{st}(y)\right|\in\R\quad\forall
x,y\in\LL\] defines a pseudo-metric having the desired
properties. Note that $\LL$ is not Dedekind complete, since the set of
all the infinitesimals is bounded but does not admit a least upper
bound.
\begin{rem}
\label{rem:P-poinInNSA} In the quotient field ${}^{*}\R_{\text{frac}}$,
the assertion of Theorem \ref{thm:infinitesimalsAndNullSequences} is
not generally true, unless one considers a particular type of
ultrafilter, called a P-point.  While the existence of a free
ultrafilter can be proved using Zorn's lemma, which is equivalent to
the axiom of choice, the existence of a P-point cannot be proved in
ZFC, that is using the usual axioms of set theory plus the axiom of
choice. Assuming the continuum hypothesis or Martin's axiom and using
transfinite induction, it is possible to prove the existence of a
P-point. See \cite{CKKR} and references therein for more details about
this foundational wrinkle.%
\footnote{The previous Theorem
\ref{thm:infinitesimalsAndNullSequences} can be easily extended to
${}^{*}\R_{\text{frac}}$ if we consider fractions $\frac{[u]}{[v]}$
for which the limit $\lim_{n\to+\infty}\frac{u_{n}}{v_{n}}$ exists
finite. It results that $\frac{[u]}{[v]}$ is infinitesimal in
${}^{*}\R_{\text{frac}}$ if and only the limit of this fraction is
zero. This permits to extend the definition of the standard part
function to all the fractions which are of the form $\frac{0}{0}$ but
whose limit exists finite. Finally, because of the previous Remark
\ref{rem:P-poinInNSA} and of the isomorphism
${}^{*}\R_{\text{frac}}\simeq{}^{*}\R$, we can always find an
infinitesimal $\frac{[u]}{[v]}\in{}^{*}\R_{\text{frac}}$ which is not
generated by an infinitesimal sequence
$\left(\frac{u_{n}}{v_{n}}\right)_{n}$; of course, this fraction is of
the form $\frac{0}{0}$ but the corresponding ratio of sequences
doesn't converge.}
\end{rem}

\section{\label{sec:Leibniz's-law}Leibniz's law of continuity in $\LL$}

To convey the full power of the idea \eqref{eq:idea_dominant}, we have
to go back to Leibniz.  Leibniz introduced infinitesimal and infinite
quantities, and developed a heuristic principle called the {}``law of
continuity'', which had roots in the work of earlier scholars such as
Nicholas of Cusa and Johannes Kepler.  It is the principle that:
\begin{quotation}
What succeeds for the finite numbers succeeds also for the infinite
numbers 
\end{quotation}
\noindent (see Knobloch \cite[p.~67]{Kn}, Robinson \cite[p.~266]{Ro66},
and Laugwitz \cite{Lau92}).

Kepler had already used it to calculate the area of the circle by
representing the latter as an infinite-sided polygon with infinitesimal
sides, and summing the areas of infinitely many triangles with infinitesimal
bases. Leibniz used the law to extend concepts such as arithmetic
operations, from ordinary numbers to infinitesimals, laying the groundwork
for infinitesimal calculus.

Of course, a modern mathematical version of this heuristic law depends
on our formalization of the first word {}`what' in the law of
continuity as stated above. We have already seen that this is almost
trivial if what we mean by the Leibnizian {}`what' is {}``continuous
equalities between real numbers''.  In fact we can extend arbitrary
continuous functions as follows.  Recall that ${\mathcal{C}}$ is the
space of Cauchy sequences of real numbers.
\begin{defn}
\label{def:extensionOfContinuousFunctions} Let $f:\R^{d}\rightarrow\R$
be a continuous function. Then we have
$f\circ(u^{1},\dots,u^{d})\in\mathcal{C}$ for every $d$-tuple of
Cauchy sequences $u^{1},\dots,u^{d}\in\mathcal{C}$, and we can define
the extension ${}^{*}\! f$ by setting \[ {}^{*}\!
f\left([u^{1}],\dots,[u^{d}]\right):=
\left[\left(f(u_{n}^{1},\dots,u_{n}^{d})\right)_{n\in\N}\right]_{\sim}\quad\forall[u^{1}],\dots,[u^{d}]\in\LL.
\]
\end{defn}
This gives a true extension of $f$,
i.e. ${}^{*}\!f(r_{1},\dots,r_{d})=f(r_{1},\dots,r_{d})$ for every
$r_{1},\dots,r_{d}\in\R$ (identified with the corresponding constant
sequences).
\begin{thm}
\label{thm:elementaryTransfer} Let $f$, $g:\R^{d}\rightarrow\R$
be continuous functions, then the equality
\begin{equation}
\forall x_{1},\dots,x_{d}\in\R:\
f(x_{1},\dots,x_{d})=g(x_{1},\dots,x_{d})
\label{eq:1ElementaryTransfer}
\end{equation}
is satisfied if and only if 
\begin{equation}
\forall\alpha_{1},\dots,\alpha_{d}\in\LL:\
{}^{*}\!f(\alpha_{1},\dots,\alpha_{d})=
{}^{*}\!g(\alpha_{1},\dots,\alpha_{d})
\label{eq:2ElementaryTransfer}
\end{equation}
Analogously, we can formulate the transfer of inequalities of the form
$f(x_{1},\dots,x_{d})<g(x_{1},\dots,x_{d})$.
\end{thm}

\begin{proof}
\noindent The equality \eqref{eq:1ElementaryTransfer} implies 
\[
\left\{ n\in\N\,|\,
f(a_{n}^{1},\dots,a_{n}^{d})=g(a_{n}^{1},\dots,a_{n}^{d})\right\}
=\N,\] where $[a^{k}]=\alpha_{k}$. The whole set $\N$ is dominant, and
therefore \eqref{eq:2ElementaryTransfer} follows. The converse
implication follows from the fact that ${}^{*}f$ and ${}^{*}g$ extend
$f$ and $g$ and from the embedding $\R\subset\LL$.\end{proof}
\begin{rem}
The reader would have surely noted that some of the limitations we
have presented can be avoided by generalizing our construction
further.  For example, the ring $\LL$ is only an integral domain and
not a field, because it is not closed with respect to pointwise
inverse, because the latter are not Cauchy sequences. Similarly, we
cannot extend a general function $f:\R^{d}\rightarrow\R$ but only
continuous functions because we need to ensure that the image sequence
is Cauchy.  However, all the ideas we have introduced up to now work
if we replace $\mathcal{C}$ by the whole of $\R^{\N}$. In this way, we
obtain a field ${}^{*}\R$ and arbitrary functions can be extended. In
the setting of the ring $\LL$ only continuous functions can be
extended and hence a continuity hypothesis has to be assumed if one
wants to use its infinitesimals. For more details, see
Goldblatt~\cite{Go}.  In the present article, we adhere to the
framework of defining a new notion of completeness so as to add new
infinitesimal points to~$\R$.  We will motivate our definition by
developing some powerful key properties of an infinitesimal-enriched
extension of $\R$.
\end{rem}

Can Leibniz's law of continuity be proved for more general properties,
e.g. for order relations or disjunctions of equality and inequality or
even more general relations?  To solve this problem, we start, once
again, from a historical consideration.

Cauchy used infinitesimals to define continuity as follows: a
function~$f$ is continuous between two bounds if for all $x$ between
those bounds, the difference $f(x+h)-f(x)$ will be infinitesimal
whenever $h$ is infinitesimal, see \cite{Ca21}.  

Such a definition tends to bewilder a modern reader, used to thinking
of $f$ as being defined for real values of the variable $x$, but now
we can think of $f(x+h)$ as corresponding to ${}^{*\!}f(x+h)$. The
function $f$ is not necessarily defined on all of $\R$, so that an
extension of the real domain~$D$ of the function is implicit in
Cauchy's construction.  Therefore, we will start by defining such an
extension of $D\subseteq\R$.  We will first define the symbol
``$\in_n$'', and then define $\ext{\!D}$ in terms of $\in_n$.
\begin{defn}
\label{def:StarExtension}Let $u\in\mathcal{C}$ be a Cauchy sequence
and $D\subseteq\R$, then 
\begin{enumerate}
\item 
$u_{n}\in_{n}D\quad\iff\quad\left\{ n\in\N\,|\, u_{n}\in D\right\}
\text{ is dominant}$
\item $\ext{\!D}:=\left\{ [u]\in\LL\,|\, u_{n}\in_{n}D\right\} $
\end{enumerate}
Let us note that the variable $n$ is mute in the notation $u_{n}\in_{n}D$.

\end{defn}
\noindent Using this notation, our questions concerning Leibniz's
law of continuity can be formulated as preservation properties of
the operator $\ext{(-)}$. In fact, as in Theorem \ref{thm:elementaryTransfer},
where equalities between continuous functions are preserved, we can
ask whether $\ext{(-)}$ preserves intersections (i.e. {}``and''),
unions (i.e. {}``or''), set-theoretic difference (i.e. {}``not''),
inclusions (i.e. {}``if... then...''), etc. To this end, it is interesting
to note that a minimal set of extension properties necessary implies
ultrafilter conditions.

We will use a circle superscript ${}^{\circ}$ in place of a star
to indicate a general extension. 
\begin{thm}
\label{thm:extensionPropsImpliesUltrafilter} Assume that ${}^{\circ}(-):\mathcal{P}(\R)\rightarrow\mathcal{P}({}^{\circ}\R)$
preserves unions, intersections and complements, i,.e. for every $A$,
$B\subseteq\R$, we have \[
{}^{\circ}\left(A\cup B\right)={}^{\circ}A\cup{}^{\circ}B\]
 \[
{}^{\circ}\left(A\cap B\right)={}^{\circ}A\cap{}^{\circ}B\]
 \[
{}^{\circ}\left(A\setminus B\right)={}^{\circ}A\setminus{}^{\circ}B.\]
 Finally, let $e\in{}^{\circ}\R$. Then \[
\mathcal{R}_{e}:=\left\{ X\subseteq\R\,|\, e\in{}^{\circ}X\right\} \text{ is an ultrafilter on }\R,\]
 and if $e\in{}^{\circ}\N$, then\[
\mathcal{N}_{e}:=\left\{ X\cap\N\,|\, X\in\mathcal{R}_{e}\right\} \text{ is an ultrafilter on }\N.\]
 \end{thm}
\begin{proof}
We need first to prove that ${}^{\circ}(-)$ preserves also the empty
set and inclusions. Indeed, ${}^{\circ}\emptyset={}^{\circ}(\emptyset\setminus\emptyset)={}^{\circ}\emptyset\setminus{}^{\circ}\emptyset=\emptyset$.
Assume $A\subseteq B$, so that $A=A\cap B$ and ${}^{\circ}A={}^{\circ}A\cap{}^{\circ}B$
and thus ${}^{\circ}A\subseteq{}^{\circ}B$.

If $X$, $Y\in\mathcal{R}_{e}$, then $e\in{}^{\circ}X\cap{}^{\circ}Y={}^{\circ}(X\cap Y)$,
and hence $X\cap Y\in\mathcal{R}_{e}$. If $X\in\mathcal{U}_{e}$
and $\R\supseteq Y\supseteq X$, then $e\in{}^{\circ}X\subseteq{}^{\circ}Y$
and hence $Y\in\mathcal{R}_{e}$. If $X\subseteq\R$, then $\R=X\cup\left(\R\setminus X\right)$;
but $e\in{}^{\circ}\R={}^{\circ}X\cup\left({}^{\circ}\R\setminus{}^{\circ}X\right)$
and therefore $X\in\mathcal{R}_{e}$ or $\R\setminus X\in\mathcal{R}_{e}$,
and this finally proves that $\mathcal{R}_{e}$ is an ultrafilter
on $\R$ because every $X\in \mathcal{R}_e$ is not empty since ${}^\circ (-)$ preserves the empty set.

The proof that $\mathcal{N}_{e}$ is closed with respect to intersection
is direct. Consider $\N\supseteq S\supseteq X\cap\N$ with $X\in\mathcal{R}_{e}$;
then $Y:=(S\setminus X)\cup X\supseteq X$ and hence $Y\in\mathcal{R}_{e}$.
Therefore, $Y\cap\N=S$ because $X\cap\N\subseteq S$, and hence $S\in\mathcal{N}_{e}$.
Finally, if $S\subseteq\N$, then either $S\in\mathcal{R}_{e}$, and
thus $S=S\cap\N\in\mathcal{N}_{e}$, or $\R\setminus S\in\mathcal{R}_{e}$.
In the second case, $(\R\setminus S)\cap\N=\N\setminus S\in\mathcal{N}_{e}$.
Up to now, we didn't need the further hypothesis $e\in{}^{\circ}\N$.
However, in this case, if $X\in\mathcal{R}_{e}$, then $e\in{}^{\circ}X\cap{}^{\circ}\N={}^{\circ}(X\cap\N)\ne\emptyset$
and hence also $X\cap\N\ne\emptyset$. 
\end{proof}
Taking, e.g., $e=1\in\ext{\N}$, yields an ultrafilter (a so-called principal ultrafilter, see the next Corollary \ref{cor:propExt}).

The meaning of this theorem is the following: if one doesn't like the
idea \eqref{eq:idea_dominant} but wants to obtain something
corresponding to Leibniz's law of continuity, one must face the
problem that the corresponding extension operator ${}^{\circ}(-)$
cannot preserves {}``and'', {}``or'' and {}``not'' of arbitrary
subsets. In Section~\ref{sec:nilpotentInf}, where we will introduce
another idea to refine Cauchy's equivalence relation without using
ultrafilters, we will see that a corresponding law of continuity
holds, but only for open subsets, so that we are forced to define a
set-theoretical difference with values in open sets\[ A\setminus
B:=\text{int}(A\setminus B),\] where $\text{int}(-)$ is the interior
operator. Note that the use of open sets and this {}``not'' operator
correspond to the semantics of intuitionistic logic.

For the sake of completeness, we also add the following results, which
represents particular cases of the previous
Theorem~\ref{thm:extensionPropsImpliesUltrafilter}.
\begin{cor}
\label{cor:propExt}
In the hypotheses of Theorem
\ref{thm:extensionPropsImpliesUltrafilter}, if\begin{equation}
X\subseteq{}^{\circ}X\ ,\ \left({}^{\circ}X\setminus
X\right)\cap\R=\emptyset\quad\forall
X\subseteq\R,\label{eq:properExtension}\end{equation} then we have
that $e\in\R$ if and only if $\mathcal{R}_{e}$ is the \emph{principal
ultrafilter generated by $e$}, i.e.\begin{equation}
\mathcal{R}_{e}=\left\{ X\subseteq\R\,|\, e\in X\right\}
.\label{eq:principalUltrafilter}\end{equation}
 \end{cor}
\begin{proof}
Let us assume that $e\in\R$ and prove the equality \eqref{eq:principalUltrafilter}.
If $e\in X\subseteq\R$, then $e\in{}^{\circ}X$ because $X\subseteq{}^{\circ}X$
by hypotheses, and therefore $X\in\mathcal{R}_{e}$. Vice versa if
$e\in{}^{\circ}X$, then ${}^{\circ}X=X\cup\left({}^{\circ}X\setminus X\right)$
and hence $e\in X$ because, by hypotheses, $\left({}^{\circ}X\setminus X\right)\cap\R=\emptyset$
and $e\in\R$.

Finally, the converse implication follows directly from the equality
\eqref{eq:principalUltrafilter} and from $\R\in\mathcal{R}_{e}$.
\end{proof}
Therefore, if the extension operation ${}^{\circ}X$ really extends $X$
(first condition of \eqref{eq:properExtension}) adding only new non real
points (second condition of \eqref{eq:properExtension}), then taking
$e\in\R$ we get a trivial ultrafilter. However, in our construction we
started from a free ultrafilter; this is the case considered in the
following corollary.
\begin{cor}
In the hypotheses of Corollary \ref{cor:propExt}, let us assume that
$\left({}^{\circ}\R,\le\right)$ is an ordered set extending the usual
order relation on the reals. Suppose that
$e\in{}^{\circ}\R\setminus\R$ is \emph{infinite} with respect to
$\left({}^{\circ}\R,\le\right)$, i.e.\[ \forall N\in\N:\ e>N\] and
also that\[ \forall N\in\N:\ e>N\ \Rightarrow\
e\in{}^{\circ}[N,+\infty),\] then the ultrafilter $\mathcal{N}_{e}$ is
free.
\end{cor}
For example, the field ${}^{*}\R_{\text{frac}}$ satisfies the
hypotheses of this corollary if we take
$e=\frac{[1]}{\left[\left(\frac{1}{n}\right)_{n}\right]}$.
\begin{proof}
By our hypothesis, every interval $[N,+\infty)=\left\{ x\in\R\,|\, x\ge N\right\} $
is in $\mathcal{R}_{e}$, therefore $[N,+\infty)\cap\N\in\mathcal{N}_{e}$.
If $X\subseteq\N$ is cofinite, then $\N\setminus X\subseteq[0,N)$
for some $N\in\N$ and hence $X\supseteq[N,+\infty)\cap\N$. From
Theorem \ref{thm:extensionPropsImpliesUltrafilter}, we have that
$\mathcal{N}_{e}$ is an ultrafilter, so that it is closed with respect
to supersets, and hence $X\in\mathcal{N}_{e}$. 
\end{proof}
Our operator $\ext{(-)}$ has the following preservation properties
of propositional logic operators.
\begin{thm}
\label{thm:transfPropositionalLogic}Let $A$, $B\subseteq\R$, then
the following preservation properties hold\end{thm}
\begin{enumerate}
\item $\ext{\left(A\cup B\right)}=\ext{A}\cup\ext{B}$ 
\item $\ext{\left(A\cap B\right)}=\ext{A}\cap\ext{B}$ 
\item $\ext{\left(A\setminus B\right)}=\ext{A}\setminus\ext{B}$ 
\item $A\subseteq B$ if and only if $\ext{A}\subseteq\ext{B}$ 
\item $\ext{\emptyset}=\emptyset$ 
\item $\ext{A}=\ext{B}$ if and only if $A=B$.\end{enumerate}
\begin{proof}
For example, we will prove the preservation of unions, the other proofs
being similar. Take $[u]\in\ext{\left(A\cup B\right)}$, then $\left\{ n\,|\, u_{n}\in A\cup B\right\} $
is dominant. If $\left\{ n\,|\, u_{n}\in A\right\} $ is dominant,
then $[u]\in\ext{A}$; vice versa, $\left\{ n\,|\, u_{n}\notin A\right\} $
is dominant and therefore it is also the intersection\[
\left\{ n\,|\, u_{n}\in A\cup B\right\} \cap\left\{ n\,|\, u_{n}\notin A\right\} =\left\{ n\,|\, u_{n}\in B\right\} ,\]
 so that $[u]\in\ext{B}$. Vice versa, if e.g. $[u]\in\ext{A}$, then
$\left\{ n\,|\, u_{n}\in A\right\} $ is dominant, and hence also
the superset $\left\{ n\,|\, u_{n}\in A\cup B\right\} $ is dominant,
i.e. $[u]\in\ext{\left(A\cup B\right)}$.\end{proof}
\begin{example}
Let $A$, $B$, $C\subseteq\R$ and write e.g. $A(x)$ to mean $x\in A$.
We want to see that our previous Theorem \ref{thm:transfPropositionalLogic}
implies that Leibniz's law of continuity applies to complicated formulas
like\begin{equation}
\forall x\in\R:\ A(x)\ \Rightarrow\ \left[B(x)\text{ and }\left(C(x)\ \Rightarrow\ D(x)\right)\right].\label{eq:ExampleTransfProp1}\end{equation}
 In other words, we will show how to apply the previous theorem to
show that \eqref{eq:ExampleTransfProp1} holds if and only if the
following formula holds\begin{equation}
\forall x\in\LL:\ \ext{A}(x)\ \Rightarrow\ \left[\ext{B}(x)\text{ and }\left(\ext{C}(x)\ \Rightarrow\ \ext{D}(x)\right)\right],\label{eq:ExampleTransfProp2}\end{equation}
 where e.g. $\ext{A}(x)$ means $x\in\ext{A}$. In fact, if we assume
\eqref{eq:ExampleTransfProp1}, this implies that $A\subseteq B$
and hence, by Theorem \ref{thm:transfPropositionalLogic}, $\ext{A}\subseteq\ext{B}$.
Therefore, if we assume $\ext{A}(x)$, for $x\in\LL$, from this we
immediately obtain $\ext{B}(x)$. The hypotheses \eqref{eq:ExampleTransfProp1}
also implies that $A\cap C\subseteq D$, so that if we further assume
$\ext{C}(x)$ we also obtain that $\ext{D}(x)$ holds, and this concludes
the proof of \eqref{eq:ExampleTransfProp2}. Analogously we can prove
the opposite implication.\end{example}
\begin{rem}
Of course, the previous example can be generalized to every logical
formula, proceeding by induction on the length of the formula, but
this requires the usual (simple) background of (elementary) formal
logic.  As it is well known (see e.g. \cite{Lu73,Hen,Be-DN,Kei}), our
example further shows that this {}``more advanced'' use of nonstandard
analysis can be left only to selected readers.
\end{rem}
Now, the next problem is natural: what about the preservation of
existential and universal quantifier? We have already considered the
case of logical connectives like {}``and'', {}``or'', {}``not''
without stressing too much the need to have a background in formal
logic. This permits to simplify our presentation and opens this type
of setting to a more general audience, including physicists and
engineers.  We wish to retain the same attitude also toward
quantifiers.  For this goal we consider two sets $X$, $Y\subseteq\R$
and the projection $p_{X}:X\times Y\rightarrow X$, $p_{X}(x,y)=x$, and
$C\subseteq X\times Y$, i.e. a relation of the form $C(x,y)$ with
$x\in X$ and $y\in Y$. We have\begin{align*} p_{X}(C) & =\left\{ x\in
X\,|\,\exists z\in C:\ x=p_{X}(z)\right\} =\\ & =\left\{ x\in
X\,|\,\exists y\in Y:\ C(x,y)\right\} ;\\ X\setminus
p_{X}\left[(X\times Y)\setminus C\right] & =\left\{ x\in
X\,|\,\neg\left(\exists y\in Y:\ (x,y)\notin C\right)\right\} =\\ &
=\left\{ x\in X\,|\,\forall y\in Y:\ C(x,y)\right\} .\end{align*}
Therefore, now our aim is to prove that $\ext{(-)}$ preserves
$p_{X}(C)$, which corresponds to the existential quantifier
(preservation of universal quantifier follows from this and from the
preservation of difference).  Only here we notice that, exactly as we
proceeded for functions considering only the continuous ones, we need
an analogous condition for a relation: what is a \emph{continuous
relation} $C\subseteq X\times Y$? To find the corresponding
definition, we start from the idea that if $f:\R\rightarrow\R$ is
continuous, then we expect that the relation $\left\{ (x,y)\in X\times
Y\,|\, y=f(x)\right\} $ is continuous. We can therefore note that the
peculiarity of the definition of the extension ${}^{*}f$ (see
Definition \ref{def:extensionOfContinuousFunctions}) is that
continuity permits us to define ${}^{*}f$ on all of $\LL$.  Otherwise,
we would always have the possibility to define ${}^{*}f$ on the
smaller domain
\[ 
\left\{ [u]\in\LL\,|\, f\circ u\in\mathcal{C}\right\} .
\] 
For this reason, we start by introducing the following definition.
\begin{defn}
\label{def:extensionOfRelation}
Let $X$, $Y\subseteq\R$ and $C\subseteq X\times Y$, then\[
\ext{C}:=\left\{
\left([u],[v]\right)\in\ext{X}\times\ext{Y}\,|\,\left(u_{n},v_{n}\right)\in_{n}C\right\}.
\] 
\end{defn} 
Next, we compare $\text{dom}(\ext{C})$ and
$\ext{\left[\text{dom}(C)\right]}$ as follows.
\begin{thm}
\label{thm:relationsDomCod}In the previous hypothesis, we always
have\[
\text{\emph{dom}}(\ext{C})\subseteq\ext{\left[\text{\emph{dom}}(C)\right]}\]
 \[
\text{\emph{cod}}(\ext{C})\subseteq\ext{\left[\text{\emph{cod}}(C)\right]},\]
 where $\text{\emph{dom}}(C)=\left\{ x\in X\,|\,\exists y\in Y:\ C(x,y)\right\} $
is the domain of $C$, and $\text{\emph{cod}}(C)=\left\{ y\in Y\,|\,\exists x\in X:\ C(x,y)\right\} $
is the codomain of $C$. \end{thm}
\begin{proof}
We prove, e.g., the relation about the domains. If $[u]\in\text{dom}(\ext{C})$,
then there exists $v$ such that $([u],[v])\in\ext{C}$, i.e. $u_{n}\in_{n}\text{dom}(C)$,
and this means that $[u]\in\ext{\left[\text{dom}(C)\right]}$. 
\end{proof}
\noindent Therefore, it is the opposite inclusion that represents
our idea of a continuous relation. 
\begin{defn}
\label{def:continuousRelation}In the previous hypothesis, we say
that: 
\begin{enumerate}
\item $C$ is continuous in the domain iff $\text{dom}(\ext{C})\supseteq\ext{\left[\text{dom}(C)\right]}$. 
\item $C$ is continuous in the codomain iff $\text{cod}(\ext{C})\supseteq\ext{\left[\text{cod}(C)\right]}$. 
\end{enumerate}
\end{defn}
For example, in the case $C=\text{graph}(f)$, the continuity in the
domain says that ${}^{*}f$ is defined on the whole $\ext{X}$. Analogously,
we can define the continuity of an $n$-ary relation with respect
to its $k$-th slot. 
\begin{thm}
If $X$, $Y\subseteq\R$, and $f:X\rightarrow Y$, then $f$ is continuous
if and only if $\text{graph}(f)$ is continuous in the domain. 
\end{thm}
The proof of this theorem can be directly deduced from the following
consideration. The continuity of $C$ in the domain can be written
as\begin{equation} \forall u\in\mathcal{C}:\
u_{n}\in_{n}\text{dom}(C)\quad\Rightarrow\quad\exists
y\in\ext{Y}:\ext{C}\left([u],y\right).\label{eq:continuousRelationDomain_1}\end{equation}
We can write this condition in a more meaningful way if we use the
following notation for an arbitrary property $\mathcal{P}(n)$:
\[
\left[\forall^{\text{d}}n:\ \mathcal{P}(n)\right]\quad
\iff\quad\left\{ n\in\N\,|\,\mathcal{P}(n)\right\} \text{ is
dominant}.
\] 
For example, $u_{n}\in_{n}D$ can now be written as
$\forall^{\text{d}}n:\ u_{n}\in D$.  Therefore,
\eqref{eq:continuousRelationDomain_1} can be written
as\begin{equation} \forall
u\in\mathcal{C}:\left(\forall^{\text{d}}n\,\exists y\in Y:\
C(u_{n},y)\right)\ \Rightarrow\ \exists y\in\ext{Y}:\
\ext{C}([u],y).\label{eq:meaningOfContinuousRelation}\end{equation}

\noindent This can be meaningfully interpreted in the following way:
if we are able to solve the equation\[
C(u_{n},y_{n})=\text{true}\]
 finding a solution $y_{n}\in Y$ for a dominant set of indices $n$,
then we are also able to solve the equation\[
\ext{C}([u],y)=\text{true}\]
 for a solution $y\in\ext{Y}$.

\noindent Using this formulation, it is not hard to prove that all
the relations $=$, $<$ and $\le$ are continuous both in the domain
and in the codomain. An expected example of non continuous relation
is $x\cdot y=1$ (take, e.g., $u_{n}:=\frac{1}{n+1}$ in \eqref{eq:meaningOfContinuousRelation}).
This corresponds to the non applicability of Leibniz's law of continuity
to the field property\[
\forall x\in\R:\ x\ne0\ \Rightarrow\ \exists y\in\R:\ x\cdot y=1,\]
 which cannot be transferred to our $\LL$, which is only a ring and
not a field.

Now, we can formulate the preservation of quantifiers: 
\begin{thm}
\label{thm:preservationQuantifiers}Let $X$, $Y\subseteq\R$ and
$C\subseteq X\times Y$ be a relation continuous in the domain, then
\[
\ext{\left[p_{X}(C)\right]}=p_{{\,}^*\!X_{\text{\rm f}}}(\ext{C}).
\]
That is
\[
\ext{\left\{ x\in X\,|\,\exists y\in Y:\ C(x,y)\right\} }=\left\{
x\in\ext{X}\,|\,\exists y\in\ext{Y}:\ \ext{C}(x,y)\right\} .\] As a
consequence we also have\[ \ext{\left\{ x\in X\,|\,\forall y\in Y:\
C(x,y)\right\} }=\left\{ x\in\ext{X}\,|\,\forall y\in\ext{Y}:\
\ext{C}(x,y)\right\} .
\]
\end{thm}
\begin{proof}
If $[u]\in\ext{\left[p_{X}(C)\right]}$, then $u_{n}\in_{n}p_{X}(C)$,
i.e.\[ \forall^{\text{d}}n:\ u_{n}\in X\ ,\ \exists y\in Y:\
C(u_{n},y),\] that is the set of $n\in\N$ satisfying this relation is
dominant.  This implies that $u_{n}\in_{n}X$ and hence $[u]\in\ext{X}$
and $u_{n}\in_{n}\text{dom}(C)$,
i.e. $[u]\in\ext{\left[\text{dom}(C)\right]}$.  Our relation $C$ is
continuous, so that $[u]\in\text{dom}(\ext{C})$, i.e.\[
\exists\beta\in\ext{Y}:\ \ext{C}([u],\beta),\] which can also be
written as 
\[
[u]\in p_{{\,}^*\!X_{\text{\rm f}}}(\ext{C}).
\]
To prove the opposite inclusion it suffices to reverse this deduction
and use Theorem \ref{thm:relationsDomCod} instead of the Definition
\ref{def:continuousRelation} of continuous relation.\end{proof}
\begin{example}
\label{exa:quantifiers}Let us apply our transfer theorems to a sentence
of the form\begin{equation} \forall a\in A\,\exists b\in B:\
C(a,b)\label{eq:ExampleTransfQuant1}\end{equation} showing that it is
equivalent to\begin{equation} \forall a\in\ext{A}\,\exists
b\in\ext{B}:\
\ext{C}(a,b),\label{eq:ExampleTransfQuant2}\end{equation} where
$C\subseteq A\times B$ is a binary continuous relation. Assume
\eqref{eq:ExampleTransfQuant1} and $a\in\ext{A}$. From
\eqref{eq:ExampleTransfQuant1} we have that\[ A\subseteq\left\{
a\in\R\,|\,\exists b\in B:\ C(a,b)\right\} ,\] Therefore, by Theorem
\ref{thm:transfPropositionalLogic} and
Theorem~\ref{thm:preservationQuantifiers}, we have
\[
\ext{A}\subseteq\left\{ a\in\LL\,|\,\exists b\in\ext{B}:\
\ext{C}(a,b)\right\}
\]
and we obtain the existence of a $b\in\ext{B}$ such that
$\ext{C}(a,b)$.  To prove the opposite implication, it suffices to
reverse this deduction.
\end{example}
The ring $\LL$ and the field ${}^{*}\R_{\text{frac}}$ can be used to
reformulate proficiently several parts of the calculus. For example, a
continuous function $f:\R\rightarrow\R$ is differentiable at $x\in\R$
if there exists $m\in\R$ such that for every non zero infinitesimal
$h\in\LL$\[ \exists\sigma\in\LL:\ f(x+h)=f(x)+h\cdot
m+h\cdot\sigma\quad,\quad\sigma\approx0,\] that is if $f(x+h)$ is equal
to the tangent line $y=f(x)+h\cdot m$ up to an infinitesimal of order
greater than $h,$ i.e. of the form $h\cdot\sigma$, with
$\sigma\approx0$. Because $\LL$ is an integral domain, taking a non
zero infinitesimal $h$, we can easily prove that such $m\in\R$ is
unique. Working in ${}^{*}\R_{\text{frac}}$, we have that such $m$ is
given by \[ m=\text{st}\left(\frac{f(x+h)-f(x)}{h}\right).\] Of
course, this real number will be denoted by $f'(x)$, so that we have
that it is infinitely close (or, in Fermat's terminology,
\emph{adequal}; see \cite[p.~28]{We84}) to the corresponding
infinitesimal ratio \[ f'(x)\approx\frac{f(x+h)-f(x)}{h}.\]
Reformulation is only the most trivial possibility offered by a
continuum with infinitesimals, because our geometrical and physical
intuition is now strongly supported by a corresponding rigorous
mathematical formalism.

\section{\label{sec:nilpotentInf}A possible approach with nilpotent infinitesimals}

\label{sub:nilpInf}

There is another approach of refining Cantor equivalence relation on
real Cauchy sequences.  This approach avoiding ultrafilters.  The idea
is to compare two sequences $u$, $v\in\mathcal{C}$ with a basic
infinitesimal, e.g. $\left(\frac{1}{n}\right)_{n}$.  We therefore set
by definition
\begin{equation} u\sim
v\quad\iff\quad\lim_{n\to+\infty}n\cdot\left(u_{n}-v_{n}\right)=0.\label{eq:FermatEquivRel}\end{equation}
In other words, using Landau's little-oh notation, the two Cauchy
sequences are to be equivalent if \[
u_{n}=v_{n}+o\left(\tfrac{1}{n}\right)\text{\;\ for\;}n\to+\infty.\]
As in the previous part of the article, we will denote the equivalence
class of a sequence $u$ simply by $[u]$. The relation defined in
\eqref{eq:FermatEquivRel} is stronger than the usual Cauchy
relation:\[ u\sim
v\quad\Rightarrow\quad\exists\lim_{n\to+\infty}u_{n}=\lim_{n\to+\infty}v_{n}=:\text{st}([u])\in\R.\]

\noindent It is also strictly stronger, because, e.g., the equivalence
class $\left[\left(\frac{1}{n^{p}}\right)_{n}\right]$, with $0<p\le1$,
is a nonzero infinitesimal. For example, the infinitesimal
$\left[\left(\frac{1}{n}\right)_{n}\right]$ is not zero, but we can
think of it as being so small that its square is zero:
$\left[\left(\frac{1}{n^{2}}\right)_{n}\right]=[0]$. With respect to
pointwise operations, we thus obtain a ring rather than a field. A
ring with nilpotent elements may seem unwieldy; however, this was
surely not the case for geometers like S. Lie, E. Cartan,
A. Grothendieck, or for physicists like P.A.M. Dirac or A. Einstein
(see, e.g., references in \cite{Gio09}). The latter used to write
formulas, if $v/c\ll1$, like \[
\frac{1}{\sqrt{1-\frac{v^{2}}{c^{2}}}}=1+\frac{v^{2}}{2c^{2}}\]
containing an equality sign rather than an approximate equality sign.
More generally, in \cite{Ein} A. Einstein wrote\begin{equation}
f(x,t+\tau)=f(x,t)+\tau\cdot\frac{\partial f}{\partial
t}(x,t)\label{eq:Einst}\end{equation} justifying it with the words
{}``\emph{since $\tau$ is very small}''.  Let us note that if we apply
\eqref{eq:Einst} to the function $f(x,t)=t^{2}$ at $t=0$, we obtain
$\tau^{2}=0+\tau\cdot0=0$ and therefore we necessarily obtain that our
ring of scalars contains nilsquare elements. Of course, it is not easy
to state that physicists like A. Einstein or P.A.M.  Dirac were
consciously working with this kind of scalars; indeed, their work,
even if it is sometimes found to be lacking from the
formal/syntactical point of view, it is always strongly supported by a
strong bridge with the physical meaning of the relationships being
discovered.

A difficult point in working with a ring having nilpotent elements is
the concrete management of powers of nilpotent elements, like
\[
h_{1}^{i_{1}}\cdot\ldots\cdot h_{n}^{i_{n}}.
\]
Let us note that this kind of product appears naturally in several
variable Taylor formulae. Is such a product zero or not? Are we able
to decide effectively whether it is zero starting from the properties
of the infinitesimals $h_{j}$ and the exponents $i_{j}$? To be able to
give an affirmative answer to this, and several other questions, we
restrict this construction to a particular subclass of Cauchy
sequences, as follows.
\begin{defn}
\label{def:little-oh_R}
We say that $u$ is a \emph{little-oh polynomial}, and we write
$u\in\R_{o}\left[\frac{1}{n}\right]$ if and only if we can write
\begin{equation}
u_{n}=r+\sum_{i=1}^{k}\alpha_{i}\cdot\frac{1}{n^{a_{i}}}+o\left(\frac{1}{n}\right)\quad\text{as}\quad
n\to+\infty,\label{eq:little-oh}\end{equation} for suitable $k\in\N$,
$r,\alpha_{1},\dots,\alpha_{k}\in\R$, $a_{1},\dots,a_{k}\in\R_{\ge0}$.
\end{defn}
Therefore, $\R_{o}\left[\frac{1}{n}\right]\subset\mathcal{C}$ and our
previous example $\left[\left(\frac{1}{n^{p}}\right)_{n}\right]$ is
generated by a little-oh polynomial. Little-oh polynomials are closed
with respect to pointwise ring operations, and the corresponding
quotient ring\[ \ER:=\R_{o}\left[\frac{1}{n}\right]/\sim\] is called
ring of \emph{Fermat reals}. The name is motivated essentially by two
reasons: in the ring of Fermat reals, it is possible to formalize the
informal method used by A.~Fermat to find maxima and minima, see
\cite{Gio11a}; all the theory of Fermat reals and Fermat extensions
has been constructed trying always to have a strong bridge between
formal properties and informal geometrical interpretation: we think
that this has been one of the leading methods used by A. Fermat in his
work.  For all the proofs of this section, we refer to
\cite{Gio10a,Gio09}.

Exactly as in the previous section about the hyperreals, we have that
the ring of Fermat reals $\ER$ is still Cauchy complete with respect
to the pseudo-metric\[
d\left(x,y\right):=\left|\text{st}(x)-\text{st}(y)\right|\in\R\quad\forall
x,y\in\ER.\] Once again, the ring $\ER$ is not Dedekind complete.

It is not hard to prove that all the numbers $k,r,\alpha_{i},a_{i}$
appearing in \eqref{eq:little-oh} are uniquely determined if we impose
upon them the constraints
\begin{equation} 0<a_{1}\le a_{2}\le\dots\le
a_{k}\le1\label{eq:constExp}\end{equation}
 \begin{equation}
\alpha_{i}\ne0\quad\forall
 i=1,\dots,k.\label{eq:constCoeff}
\end{equation} 
We can therefore introduce the following notation.
\begin{defn}
If $x:=[u]\in\ER$ and $k,r,\alpha_{i},a_{i}$ are the unique real
numbers appearing in \eqref{eq:little-oh} and satisfying \eqref{eq:constExp}
and \eqref{eq:constCoeff}, then we set $\stF{x}:=\st(x):=r$, $\stF{x}_{i}:=\alpha_{i}$,
$\omega(x):=\frac{1}{a_{1}}$, $\omega_{i}(x):=\frac{1}{a_{i}}$,
$N_{x}:=k$. Moreover, we set\[
\diff{t}_{a}:=\left[\left(\sqrt[\leftroot{7}\uproot{2} a]{\frac{1}{n}}\right)_{n}\right]\in\ER\quad\forall a\in\R_{> 0}\]
 and, more simply, $\diff{t}:=\diff{t}_{1}$. Using these notations,
we can write any Fermat real as\begin{equation}
x=\stF{x}+\sum_{i=1}^{N_{x}}\stF{x}_{i}\cdot\diff{t}_{\omega_{i}(x)}\label{eq:decomposition}\end{equation}
 where the equality sign has to be meant in $\ER$. The numbers $\stF{x}_{i}$
are called the \emph{standard parts} of $x$ and the numbers $\omega_{i}(x)$
the \emph{orders} of $x$ (for $i=1$ we will simply use the names
\emph{standard part} and \emph{order} for $\stF{x}$ and $\omega(x)$).
The unique writing \eqref{eq:decomposition} is called the \emph{decomposition
of }$x$. 
\end{defn}
Let us note the following properties of the infinitesimals of the
form $\diff{t}_{a}$:\begin{align*}
\diff{t}_{a}\cdot\diff{t}_{b} & =\diff{t}_{\frac{ab}{a+b}}\\
\left(\diff{t}_{a}\right)^{p} & =\diff{t}_{\frac{a}{p}}\quad\forall p\in\R_{\ge1}\\
\diff{t}_{a} & =0\quad\forall a\in\R_{<1}.\end{align*}

A first justification to the name {}``order'' is given by the following 
\begin{thm}
If $x\in\ER$ and $k\in\N_{>1}$, then $x^{k}=0$ in $\ER$ if and
only if $\stF{x}=0$ and $\omega(x)<k$. 
\end{thm}
This motivates also the definition of the following ideal of infinitesimals: 
\begin{defn}
If $a\in\R\cup\{\infty\}$, then\[
D_{a}:=\left\{ x\in\ER\,|\,\stF{x}=0\ ,\ \omega(x)<a+1\right\} .\]

\end{defn}
These ideals are naturally tied with the infinitesimal Taylor formula
(i.e. without any rest because of the use of nilpotent infinitesimal
increments), as one can guess from the property\[
a\in\N\quad\Rightarrow\quad D_{a}=\left\{ x\in\ER\,|\, x^{a+1}=0\right\} .\]
 Products of powers of nilpotent infinitesimals can be effectively
decided using the following result 
\begin{thm}
Let $h_{1},\dots,h_{n}\in D_{\infty}\setminus\{0\}$ and $i_{1},\dots,i_{n}\in\N$,
then 
\begin{enumerate}
\item $h_{1}^{i_{1}}\cdot\ldots\cdot h_{n}^{i_{n}}=0\ \iff\ \sum_{k=1}^{n}\frac{i_{n}}{\omega(h_{n})}>1$ 
\item $h_{1}^{i_{1}}\cdot\ldots\cdot h_{n}^{i_{n}}\ne0\ \Rightarrow\ \frac{1}{\omega\left(h_{1}^{i_{1}}\cdot\ldots\cdot h_{n}^{i_{n}}\right)}=\sum_{k=1}^{n}\frac{i_{n}}{\omega(h_{n})}$. 
\end{enumerate}
\end{thm}
This result motivates strongly our choice to restrict our construction
to little-oh polynomials only.

The reader can naturally ask what would happen in case of a different
choice of the basic infinitesimal $\left(\frac{1}{n}\right)_{n}$
in the Definition \eqref{eq:FermatEquivRel}. Really, any other choice
of a different infinitesimal $(s_{n})_{n}$ will conduct to an isomorphic
ring through the isomorphism\[
\stF{x}+\sum_{i=1}^{N_{x}}\stF{x}_{i}\cdot\diff{t}_{\omega_{i}(x)}\mapsto\left[\left(\stF{x}+\sum_{i=1}^{N_{x}}\stF{x}_{i}\cdot s_{n}^{\frac{1}{\omega_{i}(x)}}\right)_{n}\right]_{\sim}\]
 This is the only ring isomorphism preserving the basic infinitesimals
$\diff{t}_{a}$ and the standard part function, i.e. such that:\begin{align*}
f\left(\alpha\cdot\diff{t}_{a}\right) & =\alpha\cdot\left[\left(\sqrt[\leftroot{7}\uproot{2} a]{s_{n}}\right)_{n}\right]_{\sim}\\
f(\stF{x}) & =\stF{f(x)}.\end{align*}
 Essentially the same isomorphism applies also to the ring defined
in \cite{Gio10a}, where instead of sequences, the construction is
based on real functions of the form $u:\R_{\ge0}\rightarrow\R$.

\section{Order relation}

It is not hard to define an intuitively meaningful order relation on
the ring of Fermat reals
\begin{defn}
Let $x$, $y\in\ER$ be Fermat reals, then we say that $x\le y$ iff
we can find representatives $[u]=x$ and $[v]=y$ such that \[
\exists N\in\N\,\forall n\ge N:\ u_{n}\le v_{n}.\]

\end{defn}
For all the proofs of this section, see e.g. \cite{Gio11b,Gio09}.

It is not hard to show that this relation is well defined on $\ER$
and that the induced ordered relation is total. This is another strong
motivation for the choice of little-oh polynomials in the construction
of the ring of Fermat reals. The analogous of Theorem \ref{thm:infinitesimalsAndNullSequences}
is the following 
\begin{thm}
Let $h\in\ER$, then the following are equivalent 
\begin{enumerate}
\item $h\in D_{\infty}$, i.e. $\stF{h}=0$, i.e. $h$ is an infinitesimal 
\item $\forall n\in\N_{>0}:\ -\frac{1}{n}<h<\frac{1}{n}$ 
\end{enumerate}
\end{thm}

\begin{figure}%[ht]
\includegraphics[height=1.3in]{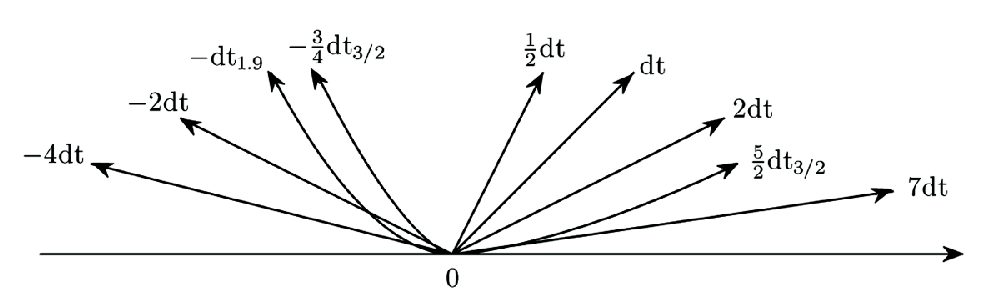} 
\caption{{\small Some first order infinitesimals}}
\label{fig:firstOrderInfinitesimals} 
\end{figure}

The following theorem permits to decide algorithmically the order
relation between two Fermat reals, using only their decompositions 
\begin{thm}
Let $x,$$y\in\ER$. If $\stF{x}\ne\stF{y}$, then\[
x<y\iff\stF{x}<\stF{y}.\]
 Otherwise, if $\stF{x}=\stF{y}$, then 
\begin{enumerate}
\item If $\omega(x)>\omega(y)$, then $x>y$ if and only if $\stF{x}_{1}>0$. 
\item If $\omega(x)=\omega(y)$, then\begin{align*}
\stF{x}_{1}>\stF{y}_{1} & \ \Rightarrow\ x>y\\
\stF{x}_{1}<\stF{y}_{1} & \ \Rightarrow\ x<y.\end{align*}

\end{enumerate}
\end{thm}

\noindent For example, $0<\diff{t}<\diff{t}_{2}<\diff{t}_{3}$, etc.
This motivates why we take $\frac{1}{a}$ in the definition of
$\diff{t}_{a}$: in this way the greater is the order $a$ and the
greater is the infinitesimal.

The ring $\ER$ can also be represented geometrically. 
\begin{defn}
If $x\in\ER$ and $\delta\in\R_{>0}$, then \begin{equation}
\text{\rm graph}_{\delta}(x):=\left\{ (\stF{x}+\sum_{i=1}^{N_{x}}\stF{x_{i}}\cdot t^{1/\omega_{i}(x)},t)\,|\,0\le t<\delta\right\} \label{eq:graph}\end{equation}

\end{defn}

E.g. $\text{graph}_{\delta}(\diff{t}_{2})=\{(\sqrt{t},t)|0\le
t<\delta\}$.  Note that the values of the function are placed in the
abscissa position, so that the correct representation of
$\text{graph}_{\delta}(x)$ is given by the figure
\ref{fig:firstOrderInfinitesimals}.  This inversion of abscissa and
ordinate in the $\text{graph}_{\delta}(x)$ permits to represent this
graph as a line tangent to the classical straight line $\R$ and hence
to have a better graphical picture. Finally, note that if $x\in\R$ is
a standard real, then $N_{x}=0$ and the $\text{graph}_{\delta}(x)$ is
a vertical line passing through $\stF{x}=x$, i.e. they are {}``ticks
on axis''.

The following theorem introduces the geometric representation of the
ring of Fermat reals.

\begin{thm}
%\noindent 
\label{thm:representationTheorem}
If $\delta\in\R_{>0}$, then the function \[
x\in\ER\mapsto\text{\emph{graph}}_{\delta}(x)\subset\R^{2}\] is
injective. Moreover if $x$, $y\in\ER$, then we can find
$\delta\in\R_{>0}$ (depending on $x$ and $y$) such that\[ x<y\] if and
only if\begin{equation} \forall p,q,t:\
(p,t)\in\text{\emph{graph}}_{\delta}(x)\text{ ,
}(q,t)\in\text{\emph{graph}}_{\delta}(y)\ \Rightarrow\
p<q\label{eq:orderInTheGeometricalRepresentation}\end{equation} that
is if a point $(p,t)$ on $\text{\emph{graph}}_{\delta}(x)$ comes
before (with respect to the order on the $x$-axis) the corresponding
point $(q,t)$ on $\text{\emph{graph}}_{\delta}(y)$.
\end{thm}

\section{Infinitesimal Taylor formula and computer implementation}

What kind of functions $f:\R\rightarrow\R$ can be extended on $\ER$?
The idea for the definition of extension is natural
$\extF{f}([u]):=[f\circ u]$ so that we have to chose $f$ so that:
\begin{enumerate}
\item If $u$ is a little-oh polynomial, then also $f\circ u$ is a little-oh
polynomial. 
\item If $[u]=[v]$, then also $[f\circ u]=[f\circ v]$. 
\end{enumerate}
The second condition is surely satisfied if we take $f$ locally
Lipschitz, but the first one holds if $f$ is smooth.
\begin{defn}
Let $f:\R^{d}\rightarrow\R$ be a smooth function, then
\[
\extF{f}([u_{1}],\dots,[u_{d}]):=[f\circ u_{1},\dots,f\circ
u_{d}]\quad\forall[u_{1}],\dots,[u_{d}]\in\ER.
\]
\end{defn}
Therefore, the ring of Fermat reals seems potentially useful e.g. for
smooth differential geometry (see e.g. chapter 13 of \cite{Gio09}) or
in some part of physics (see e.g. \cite{Gio10b}), where one can
suppose to deal only with smooth functions.

In several applications, the following infinitesimal Taylor formulae
permit to formalize perfectly the informal results frequently appearing
in physics. 
\begin{thm}
\label{thm:DerivationFormula} Let $x\in\R$ and $f:\R\rightarrow\R$
a smooth function\emph{,} then\begin{equation}
\exists!\, m\in\R\ \forall h\in D_{1}:\ f(x+h)=f(x)+h\cdot m.\label{eq:DerivationFormula}\end{equation}

\noindent In this case we have $m=f^{\prime}(x)$, where $f^{\prime}(x)$
is the usual derivative of $f$ at $x$. 
\end{thm}
\ 
\begin{thm}
\label{thm:OrdinaryTaylorFor_nVariables}Let $x\in\R^{d}$, $n\in\N_{>0}$
and $f:\R^{d}\rightarrow\R$ a smooth function\emph{,} then\[
\forall h\in D_{n}^{d}:\ f(x+h)=\sum_{\substack{j\in\N^{d}\\
|j|\le n}
}\frac{h^{j}}{j!}\cdot\frac{\partial^{|j|}f}{\partial x^{j}}(x).\]

\end{thm}

Note that~$m=f^{\prime}(x)\in\R$ in
Theorem~\ref{thm:DerivationFormula}, i.e. the slope is a standard real
number, and that we can use this formula with standard real numbers
$x$ only, and not with a generic $x\in\ER$, but it is possible to
remove these limitations (see \cite{Gio11a,Gio10b,Gio09}).

The definition of the ring of Fermat reals is highly constructive.
Therefore, using object oriented programming, it is not hard to write
a computer code corresponding to $\ER$. We (see also \cite{Gio11c})
realized a first version of this software using Matlab R2010b.

The constructor of a Fermat real is \texttt{x=FermatReal(s,w,r)},
where \texttt{s} is the $n+1$ double vector of standard parts (\texttt{s(1)}
is the standard part $\stF{x}$) and \texttt{w} is the double vector
of orders (\texttt{w(1)} is the order $\omega(x)$ if $x\in\ER\setminus\R$,
otherwise \texttt{w={[}{]}} is the empty vector). The last input \texttt{r}
is a logical variable and assumes value \texttt{true} if we want that
the display of the number \texttt{x} is realized using the Matlab
\texttt{rats} function for both its standard parts and orders. In
this way, the number will be displayed using continued fraction approximations
and therefore, in several cases, the calculations will be exact. These
inputs are the basic methods of every Fermat real, and can be accessed
using the \texttt{subsref}, and \texttt{subsasgn}, notations \texttt{x.stdParts},
\texttt{x.orders}, \texttt{x.rats}. The function \texttt{w=orders(x)}
gives exactly the double vector \texttt{x.orders} if $x\in\ER\setminus\R$
and \texttt{0} otherwise.

The function \texttt{dt(a)}, where \texttt{a} is a double, construct
the Fermat real $\diff{t}_{a}$. Because we have overloaded all the
algebraic operations, like \texttt{x+y}, \texttt{x{*}y}, \texttt{x-y},
\texttt{-x}, \texttt{x==y}, \texttt{x\textasciitilde{}=y}, \texttt{x<y},
\texttt{x<=y}, \texttt{x\textasciicircum{}y}, we can define a Fermat
real e.g. using an expression of the form \texttt{x=2+3{*}dt(2)-1/3{*}dt(1)},
which corresponds to \texttt{x=FermatReal({[}2 3 -1/3{]},{[}2 1{]},true)}.

We have also realized the function \texttt{y=decomposition(x)}, which
gives the decomposition of the Fermat real \texttt{x}, and the functions \texttt{abs(x)},
\texttt{log(x)}, \texttt{exp(x)}, \texttt{isreal(x)}, \texttt{isinfinitesimal(x)},
\texttt{isinvertible(x)}.

The function \texttt{plot(t,x)} shows the curve \eqref{eq:graph}
at the given input \texttt{t} of double.

The ratio \texttt{x/y} has been implemented if \texttt{y} is invertible.
Finally, the function \texttt{y=ext(f,x)}, corresponds to $\extF{f}(x)$
and has been realized using the evaluation of the symbolic Taylor
formula of the inline function \texttt{f}.

Using these tools, we can easily find, e.g., that\[
\frac{\sin(\diff{t}_{3}+2\diff{t}_{2})}{\cos(-\diff{t}_{4}-4\diff{t})}=\diff{t}_{3}+2\diff{t}_{2}-\frac{1}{2}\diff{t}_{\frac{6}{5}}+\frac{5}{6}\diff{t}.\]
 This corresponds to the following Matlab code:

\noindent \texttt{>\textcompwordmark{}> x=dt(3)+2{*}dt(2)}

\texttt{x = }

\texttt{dt\_3 + 2{*}dt\_2}

\noindent \texttt{>\textcompwordmark{}> y=-dt(4)-4{*}dt(1)}

\texttt{y = }

\texttt{-dt\_4 - 4{*}dt }

\noindent \texttt{>\textcompwordmark{}> g=inline('cos(y)')}

\texttt{g =}

\texttt{Inline function: g(y) = cos(y)}

\texttt{>\textcompwordmark{}> f=inline('sin(x)')}

\texttt{f =}

\texttt{Inline function: f(x) = sin(x)}

\noindent \texttt{>\textcompwordmark{}> decomposition(ext(f,x)/ext(g,y))}

\texttt{ans = }

\texttt{dt\_3 + 2{*}dt\_2 + 1/2{*}dt\_6/5 + 5/6{*}dt}

The Matlab source code is freely available under open-source licence,
and can be requested to the authors of the present article.

\section{Leibniz's law of continuity in $\ER$}

Is a suitable form of the Leibniz's law of continuity provable in the
ring of Fermat reals? The first version is the transfer for equality
and inequality, that can be proved proceeding like in Theorem
\ref{thm:elementaryTransfer}.
\begin{thm}
Let $f$, $g:\R^{d}\rightarrow\R$ be smooth functions, then it results\[
\forall x_{1},\dots,x_{d}\in\R:\ f(x_{1},\dots,x_{d})=g(x_{1},\dots,x_{d})\]
 if and only if\[
\forall x_{1},\dots,x_{d}\in\ER:\ \extF{f}(x_{1},\dots,x_{d})=\extF{g}(x_{1},\dots,x_{d}).\]
 Analogously, we can formulate the transfer of inequalities of the
form $f(x_{1},\dots,x_{d})<g(x_{1},\dots,x_{d})$. 
\end{thm}
Now, we can proceed as for $\LL$. We firstly define the extension
$\extF{U}$ of a generic subset $U\subseteq\R$. 
\begin{defn}
Define the set of little-oh polynomials $U_{o}\left[\frac{1}{n}\right]$
as in Definition \ref{def:little-oh_R} but taking sequences $u:\N\rightarrow U$
with values in $U$ and such that $\stF{[u]}:=\lim_{n\to+\infty}u_{n}\in U$.
For $u$, $v\in U_{o}\left[\frac{1}{n}\right]$ define $u\sim v$
for $u_{n}=v_{n}+o\left(\frac{1}{n}\right)$ as $n\to+\infty$ and
$\extF{U}:=U_{o}\left[\frac{1}{n}\right]/\sim$. 
\end{defn}
If $i:U\hookrightarrow\R$ is the inclusion map, it is easy to prove
that its Fermat extension $\extF{i}:\extF{U}\rightarrow\ER$ is injective.
We will always identify $\extF{U}$ with $\extF{i}(\extF{U})$, so
we simply write $\extF{U}\subseteq\ER$. According to this identification,
if $U$ is open in $\R$, we can also prove that\begin{equation}
\extF{U}=\{x\in\extF{\R}\,|\,\stF{x}\in U\}.\label{eq:FermatExtOf_U}\end{equation}
 Because of our Theorem \ref{thm:extensionPropsImpliesUltrafilter}
we must expect that our extension operator $\extF{(-)}$ doesn't preserve
all the operators of propositional logic like {}``and'', {}``or''
and {}``not''. To guess what kind of preservation properties hold
for this operator we say that the theory of Fermat reals is strongly
inspired by synthetic differential geometry (SDG; see, e.g., \cite{Kock,Mo-Re,Bel}).
SDG is the most beautiful and powerful theory of nilpotent infinitesimals
with important applycations to differential geometry of both finite
and infinite dimensional spaces. Its models require a certain knowledge
of Topos theory, because a model in classical logic is not possible.
Indeed, the internal logic of its topos models is necessarily intuitionistic.
Fermat reals have several analogies with SDG even if, at the end it
is a completely different theory. For example, in $\ER$ the product
of any two first order infinitesimals is always zero, whereas in SDG
this is not the case. On the other hand, the intuitive interpretation
of Fermat reals is stronger and there is full compatibility with classical
logic.

This background explain why we will show that our extension operator
preserves intuitionistic logical operations. Even if the theory of
Fermat reals can be freely studied in classical logic%
\footnote{More generally, without requiring a background in formal logic.%
}, the {}``most natural logic'' of smooth spaces and smooth functions
remains the intuitionistic one. We simply recall here that the intuitionistic
Topos models of SDG show formally that L.E.J. Brouwer's idea of the
impossibility to define a non smooth functions without using the law
of excluded middle or the axiom of choice is correct.

Because we need to talk of open sets both in $\R$ and in $\ER$ we
have to introduce the following 
\begin{defn}
\label{def:FermatTopology}We always think on $\ER$ the so-called
\emph{Fermat topology}, i.e. the topology generated by subsets of
the form $\extF{U}\subseteq\ER$ for $U$ open in $\R$.\end{defn}
\begin{thm}
\label{thm:transfPropIntuit}Let $A$, $B$ be open sets of $\R$,
then the following preservation properties hold 
\begin{enumerate}
\item \label{enu:1Intuiz}$\extF{(A\cup B)}=\extF{A}\cup\extF{B}$ 
\item \label{enu:2Intuiz}$\extF{(A\cap B)}=\extF{A}\cap\extF{B}$ 
\item \textup{\label{enu:3Intuiz}$\extF{\text{int}(A\setminus B)}=\text{int}(\extF{A}\setminus\extF{B})$} 
\item \label{enu:4Intuiz}$A\subseteq B$ if and only if $\extF{A}\subseteq\extF{B}$ 
\item \label{enu:5Intuiz}$\extF{\emptyset}=\emptyset$ 
\item \label{enu:6Intuiz}$\extF{A}=\extF{B}$ if and only if $A=B$ 
\end{enumerate}
\end{thm}
\begin{proof}
We will use frequently the characterization \eqref{eq:FermatExtOf_U}.
To prove \eqref{enu:1Intuiz} we have that $x\in\extF{(A\cup B)}$
iff $x\in\ER$ and $\stF{x}\in A\cup B$, i.e. iff $\stF{x}\in A$
or $\stF{x}\in B$ and, using again \eqref{eq:FermatExtOf_U}, this
happens iff $x\in\extF{A}$ or $x\in\extF{B}$. Analogously, we can
prove \eqref{enu:2Intuiz}. We firstly prove \eqref{enu:4Intuiz}.
If $A\subseteq B$ and $x\in\extF{A}$, then $\stF{x}\in A$ and hence
also $\stF{x}\in B$ and $x\in\extF{B}$. Viceversa if $\extF{A}\subseteq\extF{B}$
and $a\in A$, then $\stF{a}=a$ so that $a\in\extF{B}$, that is
$\stF{a}=a\in B$. To prove \eqref{enu:3Intuiz} we have that $x\in\extF{\text{int}(A\setminus B)}$
iff $\stF{x}\in\text{int}(A\setminus B)$, i.e. iff $(\stF{x}-\delta,\stF{x}-\delta)\subseteq A\setminus B$
for some $\delta\in\R_{>0}$. From \eqref{enu:4Intuiz} we have $\extF{(\stF{x}-\delta,\stF{x}-\delta)}\subseteq\extF{A}$
and $x\in\extF{(\stF{x}-\delta,\stF{x}-\delta)}$. Finally, a generic
$y\in\extF{(\stF{x}-\delta,\stF{x}-\delta)}$ cannot belong to $\extF{B}$
because, otherwise, $\stF{y}\in(\stF{x}-\delta,\stF{x}-\delta)\cap B$
which is impossible. Therefore, $x$ is internal to $\extF{A}\setminus\extF{B}$
with respect to the Fermat topology. The proofs of \eqref{enu:5Intuiz}
and \eqref{enu:6Intuiz} are direct or follow directly from \eqref{enu:4Intuiz}.\end{proof}
\begin{example}
Using the previous theorem, we can prove the transfer of the analogous
of \eqref{eq:ExampleTransfProp1}, but where we need now to suppose
that $A$, $B$, $C$ are open subsets of $\R$. Therefore, we have\[
\forall x\in\R:\ A(x)\ \Rightarrow\ \left[B(x)\text{ and }\left(C(x)\ \Rightarrow\ D(x)\right)\right]\]
 if and only if\[
\forall x\in\ER:\ \extF{A}(x)\ \Rightarrow\ \left[\extF{B}(x)\text{ and }\left(\extF{C}(x)\ \Rightarrow\ \extF{D}(x)\right)\right].\]

\end{example}
Once again, we don't strictly need a background of intuitionistic
logic to understand that the preservation of quantifier for the Fermat
extension $\extF{(-)}$ must be formulated in the following way 
\begin{thm}
Let $A$, $B$ be open subsets of $\R$, and $C$ be open in $A\times B$.
Let $p:(a,b)\in A\times B\mapsto a\in A$ be the projection on the
first component. Define\begin{align*}
\extF{C} & :=\left\{ (\alpha,\beta)\,|\,(\stF{\alpha},\stF{\beta})\in C\right\} \\
\exists_{p}(C) & :=p(C)\\
\forall_{p}(C) & :=\text{\emph{int}}\left(A\setminus\exists_{p}\left(\text{\emph{int}}\left(\left(A\times B\right)\setminus C\right)\right)\right).\end{align*}
 Then\begin{align*}
\extF{\left[\exists_{p}(C)\right]} & =\exists_{\extF{p}}(\extF{C})\\
\extF{\left[\forall_{p}(C)\right]} & =\forall_{\extF{p}}(\extF{C}).\end{align*}
 That is\begin{align*}
\extF{\left\{ a\in A\,|\,\exists b\in B:\ C(a,b)\right\} } & =\left\{ a\in\extF{A}\,|\,\exists b\in\extF{B}:\ \extF{C}(a,b)\right\} \\
\extF{\left\{ a\in A\,|\,\forall b\in B:\ C(a,b)\right\} } & =\left\{ a\in\extF{A}\,|\,\forall b\in\extF{B}:\ \extF{C}(a,b)\right\} .\end{align*}
 \end{thm}
\begin{proof}
The preservation of the universal quantifier follows from that of
the existential quantifier and from property \eqref{enu:3Intuiz}
of \ref{thm:transfPropIntuit}, so that we only have to prove $\extF{\left[p(C)\right]}=\extF{p}(\extF{C})$.
Consider that the projection is an open map, so that $p(C)$ is open
because $C$ is open in $A\times B$. Therefore $x\in\extF{\left[p(C)\right]}$
iff $\stF{x}\in p(C)$, and this holds iff we can find $(a,b)\in C$
such that $\stF{x}=p(a,b)=a\in A$. Therefore, $\extF{p}(x,b)=\left[\left(p(x_{n},b)\right)_{n}\right]=\left[(x_{n})_{n}\right]=x$
and $(x,b)\in\extF{C}$ because $(\stF{x},b)=(a,b)\in C$. This proves
that $\extF{\left[p(C)\right]}\subseteq\extF{p}(\extF{C})$. Vice
versa, if $x\in\extF{p}(\extF{C})$, then we can find $(\alpha,\beta)\in\extF{C}$
such that $x=\extF{p}(\alpha,\beta)=\alpha$. Therefore, $(\stF{\alpha},\stF{\beta})\in C$
and $p(\stF{\alpha},\stF{\beta})=\stF{\alpha}=\stF{x}$. This means
that $\stF{x}\in p(C)$, which is open and hence $x\in\extF{p(C)}$.\end{proof}
\begin{example}
Using the previous theorem, we can prove the transfer of the analogous
of Example \ref{exa:quantifiers}, but where we need now to suppose
that $A$, $B$ are open subsets of $\R$ and $C$ is open in $A\times B$.
Therefore, we have\[
\forall a\in A\,\exists b\in B:\ C(a,b)\]
 if and only if\[
\forall a\in\extF{A}\,\exists b\in\extF{B}:\ \extF{C}(a,b).\]

\end{example}
The theory of Fermat reals can be greatly developed: any smooth manifold
can be extended with similar infinitely closed points and the extension
functor $\extF{(-)}$ has wonderful preservation properties that generalize
what we have just seen on the (intuitionistic) Leibniz's law of continuity
in $\ER$. Potential useful applycations are in the differential geometry
of spaces of functions, like the space of all the smooth functions
between two manifolds.

\section{Conclusion}

We started with the idea of refining the equivalence relation among
real Cauchy sequences so as to obtain a new infinitesimal-enriched
continuum.  We have developed this idea in two directions.  The first
direction takes one toward the hyperreals, and we tried to motivate
the choices one must make to arrive at a powerful theory.  On the
other hand, we saw that the intuitive interpretation of such choices
is sometimes lacking.  The second idea is intuitively clearer but
surely formally less powerful.  The two ideas serve different scopes
because they deal with different kinds of infinitesimals: invertible
and nilpotent.


\begin{thebibliography}{17}

\bibitem{Bel} J.L. Bell, \emph{A Primer of Infinitesimal Analysis},
Cambridge University Press, 1998.

\bibitem{Be-DN}V. Benci and M. Di Nasso. A purely algebraic characterization
of the hyperreal numbers. \emph{Proceedings of the American Mathematical
Society}, 133(9):2501\textendash{}05, 2005.

\bibitem{BK} Borovik, A.; Katz, M.: Who gave you the
Cauchy--Weierstrass tale?  The dual history of rigorous calculus.
{\em Foundations of Science\/}, 2011, see

http://dx.doi.org/10.1007/s10699-011-9235-x

and http://arxiv.org/abs/1108.2885



\bibitem{Ca21} Cauchy, A. L.: {\em Cours d'Analyse de L'Ecole Royale
Polytechnique.  Premi\`ere Partie.  Analyse alg\'ebrique\/} (Paris:
Imprim\'erie Royale, 1821).


\bibitem{CKKR} N. Cutland, C.~Kessler, E.~Kopp, and D.~Ross. On
Cauchy's notion of infinitesimal. \emph{British J. Philos. Sci.},
39(3):375--378, 1988.

\bibitem{Ehr06} P. Ehrlich. The rise of non-archimedean mathematics
and the roots of a misconception I: The emergence of non-archimedean
systems of magnitudes. \emph{Archive for History of Exact Sciences},
60(1):1--121, 2006.

\bibitem{Ein}A. Einstein. \emph{Investigations on the Theory of
the Brownian Movement}. Dover, 1926.

\bibitem{Gio09} P. Giordano. Fermat reals: Nilpotent infinitesimals
and infinite dimensional spaces. arXiv:0907.1872, July 2009.

\bibitem{Gio10b} P. Giordano. Infinitesimals without
logic. \emph{Russian Journal of Mathematical Physics}, 17(2):159--191,
2010.

\bibitem{Gio10a} P. Giordano. The ring of fermat reals. \emph{Advances
in Mathematics}, 225(4):2050--2075, 2010.

\bibitem{Gio11a} P. Giordano. Fermat-Reyes method in the ring of
Fermat reals.  To appear in \emph{Advances in Mathematics}, 2011.

\bibitem{Gio11b} P. Giordano. Order relation and geometrical
representation of Fermat reals. Submitted to \emph{Annali della SNS},
2011.

\bibitem{Gio11c} P. Giordano and M. Kunzinger. Topological and
algebraic structures on the ring of Fermat reals. Submitted to
\emph{Israel Journal of Mathematics}, 2011.


\bibitem{Go} Goldblatt, R.: Lectures on the hyperreals.  An
introduction to nonstandard analysis. Graduate Texts in Mathematics,
188. Springer-Verlag, New York, 1998.


\bibitem{Hen}C.W. Henson. Foundations of nonstandard analysis.
A gentle introduction to nonstandard extension. In L.O. Arkeryd, N.J.
Cutland, and C.W. Henson, editors, \emph{Nonstandard analysis: theory
and applications (Edinburgh, 1996)}, pages 1\textendash{}49, Dordrecht,
1997. NATO Adv. Sci. Inst. Ser. C: Math. Phys. Sci., vol. 493, Kluwer
Acad. Publ.

\bibitem{Hew} E. Hewitt. Rings of real-valued continuous functions.
I. \emph{Trans. Amer. Math. Soc.}, 64:45--99, 1948.

\bibitem{HLO} Hrbacek, K.; Lessmann, O.; O'Donovan, R.: Analysis
with ultrasmall numbers. \emph{Amer. Math. Monthly} \textbf{117} (2010),
no. 9, 801--816.


\bibitem{KK1} Katz, K.; Katz, M.: Zooming in on infinitesimal~$1-.9..$
in a post-triumvirate era.  {\em Educational Studies in Mathematics\/}
\textbf{74} (2010), no.~3, 259-273.  See arXiv:1003.1501.


\bibitem{KK2} Katz, K.; Katz, M.: When is .999... less than 1?  {\em
The Montana Mathematics Enthusiast\/} \textbf{7} (2010), No.~1, 3--30.


\bibitem{KK11a} Katz, K.; Katz, M.: A Burgessian critique of
nominalistic tendencies in contemporary mathematics and its
historiography.  {\em Foundations of Science\/} (2011), see
http://dx.doi.org/10.1007/s10699-011-9223-1 

and http://arxiv.org/abs/1104.0375


\bibitem{KK11b} Katz, K.; Katz, M.: Cauchy's continuum.  {\em
Perspectives on Science\/} \textbf{19} (2011), no.~4, 426-452.  See
http://www.mitpressjournals.org/toc/posc/19/4 

and http://arxiv.org/abs/1108.4201


\bibitem{KK11c} Katz, K.; Katz, M.: Stevin numbers and reality, {\em
Foundations of Science\/}, 2011, see
http://dx.doi.org/10.1007/s10699-011-9228-9

and http://arxiv.org/abs/1107.3688 



\bibitem{KT} Katz, M.; Tall, D.: The tension between intuitive
infinitesimals and formal mathematical analysis.  Chapter in book
Bharath Sriraman, Editor.

http://www.infoagepub.com/products/Crossroads-in-the-History-of-Mathematics


\bibitem{Kei} J.H. Keisler. \emph{Elementary Calculus: An Approach
Using Infinitesimals}. Prindle Weber \& Schmidt, 1976.

\bibitem{Kn} E. Knobloch. Leibniz's rigorous foundation of infinitesimal
geometry by means of Riemannian sums. Foundations of the formal sciences,
1 (Berlin, 1999). \emph{Synthese} \textbf{133} (2002), no. 1-2, 59--73.

\bibitem{Kock} A. Kock. \emph{Synthetic Differential Geometry},
volume 51 of London Math. Soc. Lect. Note Series. Cambridge Univ.
Press, 1981.

\bibitem{Lau92} D. Laugwitz. Leibniz' principle and omega calculus.
{[}A{]} Le labyrinthe du continu, Colloq., Cerisy-la-Salle/Fr. 1990,
144-154 (1992).

\bibitem{Lo} J. \L{}o\'{s}. \emph{Quelques remarques, thormes et
problmes sur les classes dfinissables d'algbres}, volume Mathematical
interpretation of formal systems, pages 98--113. North-Holland Publishing
Co., Amsterdam, 1955.

\bibitem{Lu69} W.A.J. Luxemburg. Reduced powers of the real number
system and equivalents of the Hahn-Banach extension theorem. 1969
Applications of Model Theory to Algebra, Analysis, and Probability
(Internat. Sympos., Pasadena, Calif., 1967) pp. 123--137. Holt, Rinehart
and Winston, New York.

\bibitem{Lu73} W.A.J. Luxemburg. \emph{Non-Standard Analysis:
Lectures on A. Robinson's Theory of Infinitesimals and Infinitely
Large Numbers}. California Institute of Technology, Pasadena, California,
1973.

\bibitem{Mo-Re} I. Moerdijk, G.E. Reyes, \emph{Models for Smooth
Infinitesimal Analysis}, Springer, Berlin, 1991.

\bibitem{Ne} Nelson, E.: Internal set theory: a new approach to nonstandard
analysis. \emph{Bull. Amer. Math. Soc.} \textbf{83} (1977), no. 6,
1165--1198.

\bibitem{Ro66} A. Robinson. \emph{Non-standard analysis}. North-Holland
Publishing Co., Amsterdam, 1966.

\bibitem{Tar} A. Tarski. Une contribution la thorie de la mesure,
\emph{Fund. Math.} \textbf{15} (1930), 42-50.

\bibitem{We84} A. Weil. \emph{Number theory. An approach through
history. From Hammurapi to Legendre}. Birkhuser Boston, Inc., Boston,
MA, 1984. 
\end{thebibliography}
\end{document}